\documentclass[12pt,a4paper]{article}
\usepackage{amsmath,amsthm,amssymb}
\usepackage{graphicx}
\usepackage{bbm}
\usepackage{enumitem}
\usepackage[a4paper, lmargin=20mm, rmargin=20mm, tmargin=30mm, bmargin=30mm]{geometry}
\usepackage[active]{srcltx}
\usepackage{color}

\def \z{\mathbb Z}
\def \n{\mathbb N}
\def \1{\mathbbm 1}

\makeatletter
\newsavebox\myboxA
\newsavebox\myboxB
\newlength\mylenA

\renewcommand*\env@matrix[1][*\c@MaxMatrixCols c]{%
  \hskip -\arraycolsep
  \let\@ifnextchar\new@ifnextchar
  \array{#1}}

\newcommand*\xoverline[2][0.75]{%
    \sbox{\myboxA}{$\m@th#2$}%
    \setbox\myboxB\null
    \ht\myboxB=\ht\myboxA%
    \dp\myboxB=\dp\myboxA%
    \wd\myboxB=#1\wd\myboxA
    \sbox\myboxB{$\m@th\overline{\copy\myboxB}$}
    \setlength\mylenA{\the\wd\myboxA}
    \addtolength\mylenA{-\the\wd\myboxB}%
    \ifdim\wd\myboxB<\wd\myboxA%
       \rlap{\hskip 0.5\mylenA\usebox\myboxB}{\usebox\myboxA}%
    \else
        \hskip -0.5\mylenA\rlap{\usebox\myboxA}{\hskip 0.5\mylenA\usebox\myboxB}%
    \fi}
    
\newcommand\mpx[2]{
#1 \cdot\vert#1\vert_p\cdot\Vert#1#2\Vert
}   
\newcommand\length[1]{ |#1| }
    
\makeatother

\newcommand{\JB}[1]{{ #1}}

\let\phi\varphi

\newtheorem{theorem}{Theorem}[section]
\newtheorem{lemma}[theorem]{Lemma}

\newtheorem{proposition}[theorem]{Proposition}
\newtheorem{corollary}[theorem]{Corollary}
\theoremstyle{remark}
\newtheorem{remark}[theorem]{Remark}
\theoremstyle{definition}
\newtheorem{definition}[theorem]{Definition}
\newtheorem{example}[theorem]{Example}

\usepackage{mathtools}
\usepackage{textcomp}

\newcommand{\ve}{\varepsilon}

\title{A Note on the Base-$p$ Expansions of Putative Counterexamples to the $p$-adic Littlewood Conjecture}
\author{J.~Blackman, S.~Kristensen, and M.J.~Northey}

\begin{document}
\maketitle
\abstract{{In this paper, we investigate the base-$p$ expansions of putative counterexamples to the $p$-adic Littlewood {c}onjecture of de Mathan and Teuli\'e. We show that if a counterexample exists, then so does a counterexample whose base-$p$ expansion is uniformly recurrent. Furthermore, we show that if the base-$p$ expansion of $x$ is {a morphic word $\tau(\phi^\omega(a))$ where $\phi^\omega(a)$ contains a subword of the form $uXuXu$ with $\lim_{n\to\infty}|\phi^n(u)|=\infty$, then $x$ satisfies the $p$-adic Littlewood conjecture}. In the special case when $p=2$, we show that the conjecture holds for all pure morphic words.}}
\section{Introduction}

The $p$-adic Littlewood conjecture (pLC) is an open problem in Diophantine approximation, first proposed by de Mathan and Teuli\'e \cite{dMT:2004} in 2004, which states that for each prime number $p$ and all $x\in\mathbb{R}$ the following equality holds
\begin{equation}
\label{eq:pLC}
\liminf_{q \rightarrow \infty} \mpx{q}{x} = 0.
\end{equation}
Here, $\vert \cdot \vert_p$ denotes the $p$-adic absolute value, $\Vert \cdot \Vert$ denotes the distance to the nearest integer, and $q$ runs over the positive integers.  It follows trivially that if a real number $x$ is \textit{well-approximable}, \textit{i.e.}, 
$$\liminf_{q \rightarrow \infty} q\cdot{\Vert qx\Vert}=0,$$
then $x$ satisfies pLC, for all primes $p$.

In the paper that introduced this problem, de Mathan and Teuli\'e showed that (\ref{eq:pLC}) is equivalent to the condition that for each real number $x$ and all non-negative integers $k$, the partial quotients of $p^kx$ are not {uniformly bounded from above}.

\begin{lemma}\label{criterion}\emph{(\cite[Lemma~1.3]{dMT:2004})}
For each $k\in\mathbb{Z}_{\geq{0}}$, let $\overline{p^kx}=[a_{0,k};a_{1,k},\ldots]$ be the continued fraction expansion of $p^kx$. Then condition \emph{(\ref{eq:pLC})} is equivalent to
\begin{equation}\label{eq:pLC_partial_quotients}
\sup\{a_{i,k};i\geq{1},k\geq{0}\}=+\infty.
\end{equation}
\end{lemma}
 
In particular, the $p$-adic Littlewood conjecture is deeply connected to how the partial quotients of a real number behave under {iterative} prime multiplication. 
Note that since
$$\frac{1}{\sup\limits_{i\geq{1}}\left\{a_{i,k}\right\}+2}\leq \inf\limits_{q\geq{1}}\left\{ q\cdot\Vert qp^kx\Vert \right\} \leq\frac{1}{\sup\limits_{i\geq{1}}\left\{a_{i,k}\right\}},$$
 for {all} $k\in\mathbb{Z}_{\geq{0}}$ (see \cite[Ch.~7]{Bur:2000}), conditions (\ref{eq:pLC}) and (\ref{eq:pLC_partial_quotients}) are also equivalent to
\begin{equation}\label{prop:criterion}
\inf_{k \geq{0}} \inf_{q \geq{1}} q \cdot \Vert  q p^k x\Vert = 0.
\end{equation}

%

The main results regarding this conjecture can be broadly separated into two categories: 1) results which induce restrictions on the structure of the continued fraction expansions of potential counterexamples to pLC, and 2) results regarding the measure of the set of counterexamples to pLC and related objects.
Notable works regarding the continued fraction expansion of putative counterexamples to pLC include that of de Mathan and Teuli\'e \cite{dMT:2004}, which shows that quadratic irrationals satisfy pLC; Bugeaud, Drmota and de Mathan \cite{BDM:2007}, which shows that all real numbers which have arbitrarily many repetitions of a given finite block in their continued fraction expansion satsisfy pLC; and Badziahin, Bugeaud, Einsiedler and Kleinbock \cite{MR3406440}, which shows that the complexity function of the continued fraction expansion of a counterexample to pLC must grow sub-exponentially, {but the continued fraction expansion cannot be \textit{recurrent}, see below for a definition. In other words, the complexity function cannot grow too quickly or too slowly.} The main result regarding the measure of the set of potential counterexamples is that of Einsiedler and Kleinbock \cite{EK:2007}, which shows that for each prime $p$ the set of real numbers that do not satisfy (\ref{eq:pLC}) has Hausdorff dimension $0$. In fact, a stronger result was shown: this set is a countable union of sets which have box-counting dimension zero.
 
In this manuscript, instead of looking at the continued fraction expansions of potential counterexamples to pLC, we will look at the base-$p$ expansions (see Section~\ref{Section2}), which for the most part appear to have been largely unexplored.
Our main results are presented in Section~\ref{Section2}. In Section~\ref{pLC}, we look at the base-$p$ expansions of potential counterexamples to pLC and put restrictions on the type of repetitive blocks that can occur in these expansions. Furthermore, we show that if any counterexamples to pLC exist, then there exist counterexamples with uniformly recurrent base-$p$ expansions. In Section~\ref{2LC}, we utilise the results of Section~\ref{pLC} to analyse the 2-adic Littlewood conjecture. Due to the simpler alphabet, we are able to provide {stronger} results. In particular, we show that any real number with a pure morphic base-$2$ expansion satisfies $2$LC and that no counterexample to $2$LC can have arbitrarily long \textit{overlap-free} subwords -- see below. The proofs of the results of Section~\ref{pLC} are contained in Section~\ref{pLC_proofs} and the proofs for Section~\ref{2LC} are contained in Section~\ref{2LC_Proofs}.

\subsection{Notation}

Let $\mathcal{A}$ be a finite set which {we}  refer to as an \textit{alphabet} and let $\mathcal{A}^*$ be the set of all finite words over $\mathcal{A}$ including the empty word, which we denote as $\epsilon$. 
The set $\mathcal{A}^*$ forms a free monoid over $\mathcal{A}$ generated by concatenation. 
We denote the set of (right-sided) infinite {words} of $\mathcal{A}$ as $\mathcal{A}^\omega$, and denote {the} union of this set with $\mathcal{A}^*$ as $\mathcal{A}^\infty$. Given these notions, we define the \textit{length} $|\cdot|$ of a word $w\in\mathcal{A}^\infty$ to be the number of letters that appear in $w$, where $|\epsilon|=0$ and $|w|=\infty$ if $w\in\mathcal{A}^\omega$.

\begin{definition}
A finite word $w\in\mathcal{A}^*$ is an $\alpha$\textit{-power} if it can be written in the form $w=v^{\lfloor{\alpha\rfloor}}v'$ where ${\length{v'}/\length{v}\geq{\{\alpha\}:=\alpha-\lfloor{\alpha\rfloor}}}$. A word $w\in\mathcal{A}^\infty$ is \textit{overlap-free} if it contains no subword of the form $uXuXu$, where $u\in\mathcal{A}$ and $X\in\mathcal{A}^*$.
\end{definition}

{Note that a word contains {\JB{an}} overlap if and only if  it contains a subword that is a $(2+\delta)$-power for some $\delta>0$.}

\subsubsection{Morphic Words}

{An important class of words are the morphic words.  As a special case, these include all automatic words, \textit{i.e.}, words which can be generated by a finite automaton with output.}
{Let $\phi:\mathcal{A}\to\mathcal{A}^*$ be a morphism.} 

 If there is some natural number $j\geq{1}$ such that $\phi^j(a)=\epsilon$, for $a\in\mathcal{A}$, then $a$ is said to be \textit{mortal}. The set of mortal letters is denoted \JB{by} $M_\phi$. A morphism $\phi$ is \textit{prolongable} on the letter $a\in\mathcal{A}$, if $\phi(a)=ax$  and $x\not\in{M_\phi}\JB{^*}$. If a morphism is prolongable on $a$, then the words $a$, $\phi(a)$, $\phi^2(a), \ldots$ converge to an infinite word $\phi^\omega(a)$ of the form
\begin{equation}\label{morphicword}
\phi^\omega(a)=ax\cdot\phi(x)\cdot\phi^2(x)\cdot\ldots
\end{equation}
Any word that can be formed in this way is referred to as a \textit{pure morphic word}. If there is {a} coding $\tau:\mathcal{A}\to\mathcal{B}$ \JB{-- \textit{i.e.,} a morphism that maps letter to letter --} such that $w=\tau(\phi^\omega(a))$, then $w$ is referred to as a \textit{morphic word}. A morphism $\phi:\mathcal{A}\to\mathcal{A}^*$ is $k$\textit{-uniform} if $\length{\phi(a)}=k$  for all $a\in\mathcal{A}$ and is \textit{expanding} if $\length{\phi(a)}\geq{2}$ for all $a\in\mathcal{A}$. A morphism $\phi$ is \textit{primitive} if there exists some exponent $n\geq{1}$ such that for every $a,b\in\mathcal{A}$, the letter $b$ appears in the word $\phi^n(a)$ at least once.
\begin{example}
The \textit{Thue-Morse word} $M$ is the overlap-free, infinite word that is the limit $\mu^\omega(0)$ of the morphism $\mu:\{0,1\}\to\{0,1\}^*$  with $\mu(0):=01$ and $\mu(1):=10$. The first few letters are
$$M=0110100110010110\cdots.$$
{The complement of the Thue-Morse word $\widetilde{M}$ is the word given by $\mu^\omega(1)$.}
\end{example}



\section{Main Results}\label{Section2}

For every $x\in[0,1]$ and every natural number $n\geq{2}$, we can rewrite $x$ in the following form
$$x=\sum\limits_{i=1}^{\infty}a_in^{-i},$$
where $a_i\in\{0,1{,}\ldots,n-1\}$ for all $i\in\mathbb{N}$. {Unless the number $x$ is a rational number with denominator $n^k$ for some $k \ge 1$, this series expansion is unique. Since pLC is clearly satisfied for rational numbers, we will disregard this case and only consider real numbers that correspond to a unique sequence of digits.} The word formed by taking the coefficients of this power series is called the \textit{base}-$n$ \textit{expansion} of $x$. We denote this word as $w(x,n)$, \textit{i.e.}, $w(x,n):=a_1a_2\cdots$. {Conversely, g}iven a word $w\in\{0,1,\cdots,n-1\}^{\omega}$, we will denote the real number whose base-$n$ expansion coincides with $w$ as $w_n$. If $\{nx\}$ is the fractional part of $nx$, \textit{i.e.,} $\{nx\}:=nx-\lfloor{nx}\rfloor$, then the corresponding base-$n$ expansion is $T(a_1a_2a_3\cdots):=a_2a_3\cdots$. In particular, up to taking the number modulo $1$, the \textit{shift map} $T$ induces multiplication by $n$. More generally, the base-$n$ expansion of $\{n^kx\}$ corresponds to the word $T^k(a_1a_2a_3\cdots)=a_{k+1}a_{k+2}\cdots$.

Due to this structure, the base-$n$ expansion is very well-equipped for producing information regarding the limiting behaviour of a real point under repeated multiplication by $n$. Whilst the {rational approximations coming from}  the base-$n$ (or base-$p$) expansion are typically  worse
 than the {rational approximations coming from the} continued fraction expansion, in a number of cases this approximation is still good enough to induce restrictions on the potential counterexamples of pLC. On the other hand, whilst the continued fraction expansion gives a very good rational approximation of a real number, the integer multiplication of continued fractions is {far more complicated} -- see \cite{LS:98,Raney:1973}.

{For our purposes, it will also be useful to deal with \textit{base-$n$ representations} of integers. For any integer $a\geq{0}$, we can uniquely write $a$ as:
$$\sum_{i=1}^m a_in^{m-i},$$
with $a_i\in\{0,1,\ldots,n-1\}$ and $a_m\neq{0}$ (unless $m=1$).
The word $v(a,n)$ formed by taking the coefficients of this sum is the \textit{base-$n$ representation} of $a$. Given a finite, non-empty word $v$, let $v^{+}_n$ denote the integer whose base-$n$ representation coincides with $v$.}

\subsection{The $p$-adic Littlewood Conjecture}\label{pLC}

For a finite word $w$ on some alphabet $\mathcal{A}$ and a $\delta \in (0,1)$, we will denote the prefix of the word $w$ of length $\lfloor \delta \cdot\length{ w } \rfloor$ as $w^\delta$. Note that by construction, $www^\delta$ is an $\alpha$-power for all $\alpha{\leq}{2+{(\lfloor{\delta|w|\rfloor}/|w|)}}$. The following theorem shows that if the base-$p$ expansion of a real number $x$ has a {sequence of} subwords of the form  {$w_jw_jw_j^{\delta_j}$ with $\lim_{j\to\infty}\length{w_j^{\delta_j}}= \lim_{j\to\infty}\lfloor\delta_j\cdot|w_j|\rfloor=\infty$}, then $x$ satisfies pLC.

\begin{theorem}\label{Theorem:3_magic_no.}
Let $w=(a_n)_{n=1}^\infty$ be an infinite word on the alphabet $\{0,1,\dots, p-1\}$ satisfying the property that there is a sequence $(w_j)_{j=1}^\infty$ of finite words and a {sequence of positive real numbers $(\delta_j)_{j=1}^\infty$ {which are less than $1$}, } such that the word $w_j w_j w_j^{{\delta_j}}$ occurs as a subword in  $w$ {and $\lim_{j\to\infty}\length{w_j^{\delta_j}}=\infty$}. Then $w_p = \sum_{n=1}^\infty a_n p^{-n}$ satisfies the $p$-adic Littlewood conjecture.
\end{theorem}

Taking $(\delta_j)_{j=1}^\infty$ to be a constant sequence leads to the following corollary.

\begin{corollary}\label{Cor:3_magic_no.}
Assume $x$ is a counterexample to pLC and let $w(x,p)$ be the corresponding base-$p$ expansion. For each fixed $\alpha>{2}$, the length of the $\alpha$-powers appearing in $w(x,p)$ are bounded.
\end{corollary}
%


{Theorem~\ref{Theorem:3_magic_no.} can be generalised as follows.}

\begin{theorem}\label{Theorem:conjugate_pLC}
{Let $w=(a_n)_{n=1}^\infty$ be an infinite word on the alphabet $\{0,1,\dots, p-1\}$ that contains a sequence $(w_j)_{j=1}^\infty$ of finite words with $m_j=|w_j|$ and a {sequence of positive real numbers $(\delta_j)_{j=1}^\infty$} such that the word $w_jw_j^{{\delta_j}}$ occurs as a subword in  $w$.
Furthermore, let $(\ell_j)_{j=1}^\infty$ be the sequence of natural numbers satisfying
\begin{equation}\label{ineq:divisor}
p^{\ell_j-1}\leq\frac{p^{m_j}-1}{\gcd(p^{m_j}-1,(w_j)_p^+) }\leq p^{\ell_j}.
\end{equation}
  
If $\lim_{j\to\infty} m_j +\lfloor{m_j\delta_j\rfloor}-2\ell_j=\infty$, then $w_p$ satisfies pLC.
}
\end{theorem}

\noindent
{In the above theorem, the three most useful cases are:
\begin{itemize}
\item when $\gcd(p^{m_j}-1,(w_j)_p^+)=1$, $\ell_j=m_j$, and $\lim_{j\to\infty} \lfloor{m_j\delta_j\rfloor}-m_j=\infty$ (Theorem~\ref{Theorem:3_magic_no.}),
\item when $m_j=2n_j$ with $n_j\in\mathbb{N}$, $\gcd(p^{m_j}-1,(w_j)_p^+)=p^{n_j}-1$, $\ell_j=n_j+1$ and $\lim_{j\to\infty} \lfloor{m_j\delta_j\rfloor}=\infty$, and
\item when $\lim_{j\to\infty}\delta_j=\infty$.
\end{itemize}     }

{As an example of how the second of the above bullet points can be used, given a word $w=b_1b_2\cdots b_n$ in $\{0,1,\ldots,p-1\}^*$, the integer $(w\overline{w})_p^+$ will always be divisible by $p^n-1$ where $\overline{b}=p-1-b$ for letter each $b$ in the alphabet $\{0,1,\ldots,p-1\}$. This follows since $$ \sum_{i=1}^{n}p^{n-i}\cdot\left[{p^{n}b_i+p-1-b_{i}}\right]=(p^{n}-1)+\sum_{i={1}}^{n}(p^{n}-1)p^{n-i}b_i
$$
 Thus, we obtain the following corollary.
}

\begin{corollary}\label{Cor:conjugate_pLC}
Let $w=(a_n)_{n=1}^\infty$ be an infinite word on the alphabet $\{0,1,\dots, p-1\}$ satisfying the property that there is a sequence $(w_j)_{j=1}^\infty$ of finite words and a {sequence of positive real numbers $(\delta_j)_{j=1}^\infty$} such that the word $w_j \overline{w_j} w_j^{{\delta_j}}$ occurs as a subword in  $w$ {and $\lim_{j\to\infty}\length{w_j^{\delta_j}}=\infty$}. Then $w_p$ satisfies the $p$-adic Littlewood conjecture.
\end{corollary}


Another property that can be deduced is that if a word $w$ contains a sequence of increasing prefixes of another word $v$ and $v_p$ satisfies pLC, then so does $w_p$.

\begin{proposition}\label{Prop:LimitWord}
Let $w,v\in\{0,1,\ldots,p-1\}^\omega$ and assume that there exists a sequence of prefixes $(v_k)_{k=1}^\infty$ of $v$ such that $\length{v_k}\to\infty$ and $v_k$ appears as a subword of $w$ for all $k$. If $v_p$ satisfies pLC, then so does $w_p$.
\end{proposition}

{An infinite word $w = (a_n)_{n=1}^\infty$ is said to be \emph{recurrent} if any finite subword $v$ of $w$ occurs infinitely often in $w$. It is said to be \emph{uniformly recurrent} if for every  finite subword $v$ of $w$, there exists a constant $N_v$ such that $v$ appears in every subword of $w$ of length $N_v$.} Using \JB{an idea similar} to the work of Badziahin \cite{Badziahin:2014aa} on ``limit words'' of continued fraction expansions, we can look at the topological closure of the set of base-$p$ expansions of the counterexamples to pLC under the action of the shift map. This allows us to  deduce that if this set is non-empty then it contains an element with a uniformly recurrent base-$p$ expansion.

\begin{theorem}\label{Theorem:unif_rec_CE}
If there is a counterexample to pLC, there is a counterexample with a uniformly recurrent base-$p$ expansion.
\end{theorem}

\begin{remark}
It is worth noting that none of the above statements rely on $p$ being prime other than to link to the $p$-adic Littlewood conjecture. In particular, we can replace $p$ with a composite number $n$ to obtain analogous results on the ``$n$-adic Littlewood conjecture''.
\end{remark}

The proof of Theorems~\ref{Theorem:3_magic_no.} and~\ref{Theorem:conjugate_pLC} can be found in Section~\ref{Proof:increasing_factors}. The proof of Proposition~\ref{Prop:LimitWord} and Theorem~\ref{Theorem:unif_rec_CE} is in Section~\ref{Proof:Subwords}.

\subsubsection{Results on Morphic Words}

Let $w=\phi^\omega(a)$ be a pure morphic word.  If the prefix $\phi^k(a)$ contains overlap of the form $uXuXu$ for some $k\in\mathbb{N}$, then $\phi^n(u)\phi^n(X)\phi^n(u)\phi^n(X)\phi^n(u)$ is a subword of $\phi^{k+n}(a)$ for all $n\in\mathbb{N}$. Under the assumption that $u$ is not mortal for $\phi$, infinitely many instances of overlap occur. Furthermore, if $\lim_{n\to\infty}\length{\phi^n(u)}=\infty$, the word satisfies the conditions of Theorem~\ref{Theorem:3_magic_no.}. This leads to the following proposition.

\begin{proposition}\label{Proposition:prim_or_uniform}
{Let $w=\phi^\omega({a})\in\mathcal{A}^\omega$ be a pure morphic word containing a subword $uXuXu$ such that $\lim_{n\to\infty}\length{\phi^n(u)}=\infty$. For any non-erasing morphism $g:\mathcal{A}\to\{0,1,\ldots,p-1\}$, the real number $g(w)_p$ satisfies the $p$-adic Littlewood conjecture.}
\end{proposition}


\begin{remark}
{Here we should note that the condition $\lim_{n\to\infty}\length{\phi^n(u)}=\infty$ is instantly satisfied for morphisms which are expanding, including (powers of) primitive morphisms and $k$-uniform morphisms for $k\geq{2}$. Furthermore, due to a result of Durand \cite{Durand:2013}, all uniformly recurrent morphic words are primitive morphic. Therefore, if $x$ is a counterexample to pLC with a morphic uniformly recurrent base-$p$ expansion of the form $\tau(\phi^\omega(a))$, then the underlying pure morphic word $\phi^\omega(a) $ must be overlap-free.}
\end{remark}

{Similar to the previous argument, if a morphism $\phi$ is prolongable on the letters $a,b\in\mathcal{A}^*$ and $b$ appears in the word $\phi^\omega(a)$ at least once, then every prefix of $\phi^\omega(b)$ appears in $\phi^\omega(a)$. Proposition~\ref{Prop:LimitWord} then directly implies the following corollary.}

\begin{corollary}\label{Corollary:Morphic_subalphabets}
Let $w=\phi^\omega(a)$ be a pure morphic word over $\mathcal{A}$ and let $\mathcal{B}$ be a sub-alphabet of $\mathcal{A}$ such that $\phi:\mathcal{B}\to\mathcal{B}^*$. Furthermore, assume that $\phi^\omega(a)$ contains a letter $b\in\mathcal{B}$ such that $\phi$ is prolongable over $b$ and let $\tau:\mathcal{A}\to\{0,1,\ldots,p-1\}$ be a coding. If $\tau(\phi^\omega(b))_p$ satisfies pLC, then so does $\tau(\phi^\omega(a))_p$.
\end{corollary}
%
%
%
%

%
%

\subsection{Applications to the 2-adic Littlewood Conjecture}\label{2LC}

In the case of the $2$-adic Littlewood conjecture, all pure morphic words satisfy at least one of three properties \textbf{(P1)-(P3)} -- see Lemma~\ref{Lemma:BinaryPureMorphic} below. Combining this result with Theorem~\ref{Theorem:3_magic_no.} and other results in the literature leads to the following theorem.

\begin{theorem}\label{Theorem:pure_morphic_2LC}
Let $x\in[0,1]$ and assume that the corresponding base-$2$ expansion $w(x,2)$ is a pure morphic word. Then $x$ satisfies $2$LC.
\end{theorem}

This theorem can be extended to a class of results regarding pLC by applying Corollary~\ref{Corollary:Morphic_subalphabets}.

\begin{corollary}\label{Cor:pure_morphic_2LC}
Let $w=\phi^\omega(a)$ be a pure morphic word over $\mathcal{A}$ and let $\mathcal{B}$ be a sub-alphabet  of $\mathcal{A}$ such that $\phi:\mathcal{B}\to\mathcal{B}^*$ and $|\mathcal{B}|=2$. Furthermore, assume that $\phi^\omega(a)$ contains a letter $b\in\mathcal{B}$ such that $\phi$ is prolongable over $b$.
{Then $\tau(w)_p$ satisfies pLC for any coding $\tau:\mathcal{A}\to\{0,1,\ldots,p-1\}$}.
\end{corollary}

Finally, as \JB{a result contrasting with} Corollary~\ref{Cor:3_magic_no.}, we show that \JB{ the lengths} of the overlap-free subwords of the base-$2$ expansion of a counterexample to $2$LC are bounded.

\begin{theorem}\label{Theorem:bounded_overlap_free_2LC}
Assume that $x$ is a counterexample to $2$LC and let $w(x,2)$ be the corresponding base-$2$ expansion. Then the length of the overlap-free subwords in $w(x,2)$ are bounded.
\end{theorem}

The proof of Theorem~\ref{Theorem:pure_morphic_2LC} and Corollary~\ref{Cor:pure_morphic_2LC} can be found in Section~\ref{Proof:pure_morphic_2LC} and the proof of Theorem~\ref{Theorem:bounded_overlap_free_2LC} can be found in Section~\ref{Proof:bounded_overlap_free_2LC}.

\section{The $p$-adic Littlewood Conjecture}\label{pLC_proofs}

\subsection{Proof of Theorems~\ref{Theorem:3_magic_no.} and~\ref{Theorem:conjugate_pLC}}\label{Proof:increasing_factors}

To prove Theorems~\ref{Theorem:3_magic_no.} and~\ref{Theorem:conjugate_pLC},  we will show that the conditions of these theorems imply (\ref{prop:criterion}).
To this end, we will produce sequences $(q_j)_{j=1}^\infty$ and $(k_j)_{j=1}^\infty$ of natural numbers such that
\begin{equation}\label{limprop:criterion}
\lim_{j \rightarrow \infty} q_j\cdot \Vert  q_j p^{k_j} x\Vert = 0.
\end{equation}

\begin{proof}[Proof of Theorem~\ref{Theorem:3_magic_no.}]
%
For each $j \in \n$, let $k_j$ be the length of the prefix of $(a_n)_{n=1}^\infty$ up to the first occurrence of the subword $w_j w_j w_j^{\delta{_j}}$. Set 
\[
x' := \left\{p^{k_j} x\right\} = \left\{p^{k_j} \sum_{n=1}^\infty a_n p^{-n}\right\} = \sum_{n=1}^\infty a_{k_j + n} p^{-n}.
\]
Then, the base-$p$ expansion of $x'$ begins with the subword $w_j w_j w_j^{\delta{_j}}$.

Now, for each $j$, we denote $w_j$ as $ b^{(j)}_1 b_2^{(j)} \cdots b_{m_j}^{(j)}$ {where $m_j=\length{w_j}$,} and define a sequence of rational numbers 
\[
\frac{r_j}{q_j} := \sum_{h=0}^\infty \sum_{i=1}^{m_j} \frac{b_i^{(j)}}{p^{i+hm_j}} = \sum_{h=0}^\infty \frac{1}{p^{hm_j}} \sum_{i=1}^{m_j} \frac{b_i^{(j)}}{p^i}.
\]
These are the rational numbers (in reduced form) whose base-$p$ expansion is obtained by extending the word $w_j$ periodically. The sequence of denominators $(q_j)_{j=1}^\infty$ is the sequence used in (\ref{limprop:criterion}).

The numbers $r_j/q_j$ approximate $x'$ rather well. Indeed,
\[
\left\vert x' - \frac{r_j}{q_j}\right\vert = {\left\vert \sum_{i=\lfloor\delta{_j} m_j \rfloor+ 1 }^{m_j} \frac{c_{2,i}^{(j)}}{p^{i+2m_j}} +   \sum_{h=3}^\infty \frac{1}{p^{hm_j}} \sum_{i=1}^{m_j} \frac{c_{h,i}^{(j)}}{p^i} \right\vert} < \frac{1}{p^{2m_j+\lfloor\delta{_j} m_j \rfloor}},
\]

{where $c_{h,i}^{(j)}=(a_{k_j+hm_j+i}-b_i^{(j)})$}.
On the other hand,
\[
\frac{r_j}{q_j} = {\left(\sum_{h=0}^\infty \frac{1}{p^{hm_j}}\right)} {\left(\sum_{i=1}^{m_j} \frac{b_i^{(j)}}{p^i}\right)} = \frac{p^{m_j}}{p^{m_j}-1} \sum_{i=1}^{m_j}  \frac{b_i^{(j)}}{p^i} = \frac{r'_j}{p^{m_j}-1},
\]
where $r'_j \in \z$. Consequently, $q_j \le p^{m_j} - 1 < p^{m_j}$ and therefore,
\[
q_j \cdot \Vert  q_j p^{k_j}x\Vert \le q_j^{2} \cdot \left\vert x' - \frac{r_j}{q_j}\right\vert < \frac{1}{p^{\lfloor\delta{_j} m_j \rfloor}}.
\]
{Since $\delta_j\cdot m_j$} tends to infinity with $j${, the theorem follows}.
\end{proof}

The above proof illustrates a very useful technique for using combinatorial properties of base-$p$ expansions to show that real numbers satisfy pLC. The proof of Theorem~\ref{Theorem:conjugate_pLC} {serves as generalisation of the above method}.

{\begin{proof}[Proof of Theorem~\ref{Theorem:conjugate_pLC}]
For each $j\in\mathbb{N}$, let $k_j$ be the length of the prefix of $(a_n)^\infty_{n=1}$ up to the first occurrence of the subword $w_jw_j^{\delta_j}$ and set
$$x':=\{p^{k_j}x\}$$
Then, the base-$p$ expansion of $x'$ begins with the subword $w_jw_j^{\delta_j}$. Let $n_j=\lfloor{\delta_j}\rfloor$.

For each $j\in\mathbb{N}$, we denote $w_j$ as $ b^{(j)}_1 b_2^{(j)} \cdots b_{m_j}^{(j)}$ and set
\begin{equation}\label{eq:rational_approx2}
 \frac{r_j}{q_j}:=\sum_{h=0}^\infty\sum_{i=1}^{m_j}\frac{b_i^{(j)}}{p^{i+hm_j}}={\left(\sum_{h=0}^\infty\frac{1}{p^{(h+1)m_j}}\right)\left(\sum^{m_j}_{i=1}p^{{m_j}-i}b_i^{(j)}\right)}
\end{equation} 
These are the rational numbers (in reduced form) whose base-$p$ expansion is obtained by extending the word $w_j$ periodically.
 
As in the proof of Theorem~\ref{Theorem:3_magic_no.}, this sequence of rational numbers $\frac{r_j}{q_j}$ approximates $x'$ very well,

\[
\left\vert x' - \frac{r_j}{q_j}\right\vert= {\left\vert \sum_{i=\lfloor (\delta_j-n_j)m_j \rfloor+ 1 }^{m_j} \frac{c_{n_j,i}^{(j)}}{p^{i+n_jm_j}} +   \sum_{h=n_j+1}^\infty \frac{1}{p^{hm_j}} \sum_{i=1}^{m_j} \frac{c_{h,i}^{(j)}}{p^i} \right\vert} < \frac{1}{p^{m_j+\lfloor{\delta_jm_j\rfloor}}},
\]
{where $c_{h,i}^{(j)}=(a_{k_j+2hm_j+i}-b_i^{(j)})$}.

Let $d_j=\gcd(p^{m_j}-1,(w_j)_p^+)$. Then there exists some $a\in\mathbb{N}$ such that
$$ ad_j=\sum^{m_j}_{i=1}p^{{m_j}-i}b_i^{(j)}
$$


Combining this with \eqref{eq:rational_approx2}  shows 
\begin{align*}
\frac{r_j}{q_j}&=a\cdot\frac{d_j}{p^{m_j}-1}.
\end{align*}

From $\eqref{ineq:divisor}$, it follows that $q_j\leq{(p^{m_j}-1)}/{d_j}\leq p^{\ell_j}$ and therefore
$$q_j^2\cdot\left|x'-\frac{r_j}{q_j}\right|< \frac{p^{2\ell_j}}{p^{m_j+\lfloor{\delta_j m_j\rfloor}}}=\frac{1}{p^{m_j+\lfloor{\delta_j m_j\rfloor}-2\ell_j}}.$$


%
%
%
Since it was assumed that $\lim_{j\to\infty} m_j+\lfloor \delta_j m_j\rfloor -2\ell_j=\infty$, this completes the proof. 
\end{proof}}

{\subsection{Proof of Proposition~\ref{Prop:LimitWord} and Theorem~\ref{Theorem:unif_rec_CE}}\label{Proof:Subwords}
%

Given an infinite word $w\in\mathcal{A}^\omega$, we define the \textit{set of suffixes} $\mathcal{S}(w)$ of $w\in\mathcal{A}^\omega$ to be
 $$\mathcal{S}(w):=\{T^k(w): k\in\mathbb{Z}_{\geq{0}}\}.$$ We can turn $\mathcal{A}^\omega$ into a metric space by defining a metric $d(x,y)=2^{-|u|}$, where $u$ is the largest common prefix of $x$ and $y$ and $d(x,x)=0$. From this, we can take the topological closure of the set of suffixes $\overline{\mathcal{S}(w)}$.  A word $v\in\mathcal{A}^\omega$ is an element of $\overline{\mathcal{S}(w)}$ if and only if every prefix of $v$ appears in $w$. Analogously, for any $x\in[0,1]$, we can define the set $$T_p(x):=\{\{p^n x\}  : n\in\n\cup\{0\}\}.$$ Assuming $x$ is not a rational number with denominator equal to $p^k$ for natural number $k\geq{1}$, the sets $T_p(x)$ and $\mathcal{S}(w(x,p))$ are in bijection, where each real number corresponds to its base-p expansion. Likewise, the topological closures $\overline{\mathcal{T}_p(x)}$ (using the Euclidean metric) and $\overline{\mathcal{S}(w(x,p))}$  are also in bijection. This comes from the observation that there is a subsequence $\{p^{k_j}x\}$ that limits to $y$ if and only if the base-$p$ expansions of $\{p^{k_j}x\}$ limit to the base-$p$ expansion of $y$. Using the notions above, the proof of Proposition~\ref{Prop:LimitWord} essentially comes down to showing that if any accumulation point of $T_p(x)$  satisfies pLC, then $x$ satisfies pLC. The contrapositive of this statement is shown in the next lemma.

\begin{lemma}\label{T1}
Let $x$ be a counterexample to pLC and assume that there exists some $\ve>0$ such that $m_p(x)\geq{\ve}.$ Then $m_p(y)\geq\ve$ for all $y\in\overline{T_p(x)}$ .
\end{lemma}

\begin{proof}
We trivially have that $m_p(x)\geq{\ve}$ implies that $m_p(\{p^nx\})\geq{\ve}$ for every $n\in\n\cup\{0\}$, since $$\liminf_{q\to\infty} \mpx{q}{p^nx}=\liminf_{q\to\infty}\mpx{p^nq}{x}\geq \liminf_{q\to\infty} \mpx{q}{x}.$$

Assume that $y$ is a limit point of $T_p(x)$ such that $m_p(y)<\ve$. Then there exists some {$\ve'\in\mathbb{R}$ satisfying $0<\ve'<\ve<1$} and some sequence $(q_n)_{n=1}^\infty$ such that 
$$\lim\limits_{n\to\infty}\mpx{q_n}{y}=\ve'.$$ 
For every $n\in\n$, let $\delta_n=2^{-1}q_n^{-2}(\ve-\ve')$ and let $k_n$ be the smallest natural number such that $\{p^{k_n}x\}=y+\Delta_n$ with $|\Delta_n|<\delta_n$. The existence of $k_n$ follows from the fact that $y$ is an accumulation point.  This implies 
\begin{align*}
    \mpx{p^{k_n}q_n}{x}&=\mpx{q_n}{p^{k_n}x}\\
     &=\mpx{q_n}{(y+\Delta_n)}\\
    &\leq \mpx{q_n}{y}+\mpx{q_n}{\Delta_n} \\
    &\leq \mpx{q_n}{y}+q_n\cdot|q_n\Delta_n| 
\end{align*}
Then 
$$
    \lim\limits_{n\to\infty}\mpx{p^{k_n}q_n}{x} \leq{\ve'+2^{-1}(\ve-\ve')}<\ve
$$

Therefore, $m_p(x)<{\ve}$ which is a contradiction.
\end{proof}

 We can use the above lemma to deduce that should a counterexample $x$ to pLC exist, there exists an element $y$ of  $\overline{T_p(x)}$ with a uniformly recurrent base-$p$ expansion that is a counterexample to pLC. This proves Theorem~\ref{Theorem:unif_rec_CE}.

\begin{proposition}
Let $x$ be a counterexample of pLC. Then $\overline{T_p(x)}$ contains a counterexample of pLC with a uniformly recurrent base-$p$ expansion.
\end{proposition}

\begin{proof}
By construction $\overline{T_p(x)}$ is closed and bounded. Therefore, $\overline{T_p(x)}$ is compact and invariant under multiplication by $p$. 
The corresponding set of base-$p$ expansions is given by $\overline{S(w(x,p))}$ and is also compact and invariant under the shift map $T$.
At least one minimal, invariant, compact subset $R$ of $\overline{S(w(x,p))}$ exists, and by \cite[Theorem 1.5.9]{Lothaire}, this is a set  comprised of numbers with uniformly recurrent base-$p$ expansions. By Lemma~\ref{T1}, all elements in $R$ are counterexamples to pLC.
\end{proof} }

\section{The $2$-adic Littlewood Conjecture}\label{2LC_Proofs}

\subsection{Proof of Theorem~\ref{Theorem:pure_morphic_2LC}}\label{Proof:pure_morphic_2LC}

In order to prove Theorem~\ref{Theorem:pure_morphic_2LC}, it will be useful to  first introduce a number of auxiliary results. The first result is that of Se\'ebold \cite{Seebold:1985}, which shows that the only pure morphic words over $\{0,1\}$ which are overlap-free are the Thue-Morse word $M$ and its complement $\widetilde{M}$.

\begin{theorem}[\cite{Seebold:1985}]\label{Theorem:Seebold}
$M$ and $\widetilde{M}$ are the only pure morphic overlap-free words in $\{0,1\}^\omega$. 
\end{theorem}

Using this theorem, we can give the following characterisation of all binary pure morphic words.

\begin{lemma}\label{Lemma:BinaryPureMorphic}
Let $w$ be a pure morphic word in $\{0,1\}^\omega$, where $\phi$ is the underlying morphism. Then \JB{(at least) one of the following statements holds}:
\begin{enumerate}[label=\em{\textbf{(P\arabic*)}}]
\item $w$ is $M$ or $\widetilde{M}$.

\item There is a non-trivial subword $v$ of $w$, such that $v^n$ is a subword of $w$ for all $n\in\mathbb{N}$.
\item {$w$ contains overlap of the form $aXaXa$ \JB{where $a\in\mathcal{A}$ and $X\in\mathcal{A}^*$} and $\lim_{n\to\infty}|\phi\JB{^n}(a)|=\infty$.}
\end{enumerate}

\end{lemma}

\begin{proof}
In the case that $w$ is overlap-free, Theorem~\ref{Theorem:Seebold} shows that $w$ is the Thue-Morse word $M$ or its complement $\widetilde{M}$ \textbf{(P1)}.

Assume that $w$ is not overlap-free and that $w=\phi^\omega(0)$ -- the case that $w=\phi^\omega(1)$ follows by symmetry. Then, for some words $u,v\in\{0,1\}^*$, we have
\begin{equation*}\label{generalform}
\phi(0)=0u \quad \text{and} \quad \phi(1)=v{.}
\end{equation*}
Since $\phi$ is prolongable over $0$, the word $u$ is not the empty word. In particular, $\length{\phi(0)}\geq{2}$. Note that if $u$ consists only $0$'s,\textit{ i.e.}, $u=0^n$ for $n\in\mathbb{N}$, then
$$w=\phi^\omega(0)=0^{{\omega}},$$ 
{where $x^\omega$ is the \textit{periodic} word $xxx\cdots$.} 
Thus, $w$ satisfies \textbf{(P2)}.

\noindent
\textbf{Case I: $v$ {is the empty word $\epsilon$}.}
\vspace{2mm}

 Since $\phi(1)=\epsilon$, applying the morphism to $\phi^k(0)$ will ignore any $1$'s in this sequence. In other words, if $i_k$ is the number of $0$'s that appear in $\phi^k(0)$, then
$$\phi^{k+1}(0)=\left(\phi(0)\right)^{i_k}.$$
Therefore, $i_{k+1}=i_k\cdot i_1=i_1^{k+1}$.
Since $\phi$ is prolongable, $u$ contains the letter $0$ at least once, and so $i_1\geq{2}$. Since $\lim_{k\to\infty}i_k=\infty$,
$$w=\phi^\omega(0)={\left(\phi(0)\right)^{\omega}}.$$
{In this case}, $w$ satisfies \textbf{(P2)}.

\noindent
\textbf{Case II: $v=1^n$.}
\vspace{2mm}

As discussed above, we can assume that $u$ contains the letter $1$ at least once.
If $\phi(1)=1^n$ for some $n\geq{2}$, then $\phi^k(1)=1^{n^{k}}$. Since $\phi(0)$ contains the letter $1$, the word $\phi^{k+1}(0)$ contains the subword $\phi^k(1)$ for all $k\in\mathbb{N}$. Therefore, $w=\phi^\omega(0)$ satisfies \textbf{(P2)}.\\

{\JB{Let $v=1$. Note that if $u$ does not contain the letter $0$, \textit{i.e.}, $u=1^k$, then $$\phi^2(0)=\phi(0)\phi(1^k)=01^{2k},\quad \phi^3(0)=\phi(0)\phi(1^{2k})=01^{3k},  \text{ and } \phi^m(01^k)=01^{mk}.$$ In this case, $\phi^\omega(0)=01^\infty$ and \textbf{(P2)} is satisfied.  Furthermore, if $u=u'01^k$ with $k\in\mathbb{N}$, we note that for all $m\in\mathbb{N}$ $$\phi(01^m)=0u'01^k1^m=0u'01^{(k+m)}.$$ In particular, for all $n\in\mathbb{N}$ the word $\phi^n(u)$ ends in the term $1^{(n+1)k}$. Therefore, $w$ satisfies \textbf{(P2)}, and so we have now reduced our considerations to the cases where $u$ ends in the letter $0$.

If $u=1^k0$ with $k\in\mathbb{N}$, then 
$$
\phi^2(0)=\phi(0)1^k\phi(0)=(01^k0)1^k(01^k0).
$$
Since this word contains the subword $01^k01^k0$ and $\phi$ is prolongable on $0$, the length of $\phi^n(0)$ tends to infinity and \textbf{(P3)} is satisfied.}}

Finally, assume that $u=u'01^k0$ with $k\in\mathbb{Z}_{\geq{0}}$. The word
$$\phi(01^k0)=0u'01^k01^k0u'01^k0$$ contains $01^k01^k0$ as a subword. Since $\phi$ is prolongable on $0$, the length of $\phi^n(0)$ tends to infinity and \textbf{(P3)} is satisfied.
\vspace{2mm}

\noindent
\textbf{Case III: $v$ contains $0$.}
\vspace{2mm}

Again, we can freely assume that $u$ contains the letter $1$. Since $v$ contains the letter $0$, the morphism $\phi$ is primitive: if $v$ contains both $0$ and $1$, it follows by definition; if $v$ only contains the letter $0$, then $\phi^2(1)$ will contain $\phi(0)$ which contains both the letters $0$ and $1$. {Since $\phi$ is primitive $\lim_{n\to\infty}|\phi^n(0)|=\lim_{n\to\infty}|\phi^n(1)|=\infty$.}  These words {contain overlap by previous assumption and, therefore,}  satisfy  \textbf{(P3)}.
\end{proof}

The final result needed to prove Theorem~\ref{Theorem:pure_morphic_2LC} is due to Badziahin and Zorin, which shows that the real number that has the Thue-Morse word (or its complement) as its base-$n$ expansion is well-approximable provided that $n$ is not divisible by $15$. Note that if $n$ is divisible by $15$ the result is unknown, as opposed to being false.

\begin{theorem}[\cite{BZ:2014}]\label{Theorem:TM_bad_approx}
Let $M_n$ be the real number whose base-$n$ expansion is the Thue-Morse word. If {$n$ is not divisible by $15$}, then $M_n$ is well-approximable.
\end{theorem}

Combining together Proposition~\ref{Proof:Subwords}, Lemma~\ref{Lemma:BinaryPureMorphic}, and Theorem~\ref{Theorem:TM_bad_approx} provides the proof for Theorem~\ref{Theorem:pure_morphic_2LC}.

\begin{proof}[Proof of Theorem~\ref{Theorem:pure_morphic_2LC}]
From Theorem~\ref{Theorem:TM_bad_approx}, $M_2$ is well-approximable and therefore, satisfies 2LC. In this case, $\widetilde{M}_2$ is given by $1-M_2$. Since $M_2$ is well-approximable, so is $\widetilde{M}_2$. Therefore, the real numbers whose base-$2$ expansion  satisfy $\textbf{(P1)}$ satisfy $2$LC. For  words satisfying \textbf{(P2)}, we note that for any periodic word $v$, \textit{i.e.}, $v={X^\omega}$, the real number $v_2$ is rational and therefore, well-approximable. Applying Proposition~\ref{Prop:LimitWord} shows that the real numbers whose base-$2$ expansions satisfy \textbf{(P2)}, also satisfy $2$LC. Finally, Proposition~\ref{Proposition:prim_or_uniform} implies that for any base-$2$ expansion which satisfies \textbf{(P3)}, the corresponding real number satisfies $2$LC.
\end{proof}

\subsubsection{Proof of Corollary~\ref{Cor:pure_morphic_2LC}}

From Corollary~\ref{Corollary:Morphic_subalphabets}, we can extend Theorem~\ref{Theorem:pure_morphic_2LC} to Corollary~\ref{Cor:pure_morphic_2LC}, by showing that for any morphism $\psi:\{a,b\}\to\{a,b\}^*$ which is prolongable on $a$ with $a,b\in\mathcal{A}$ and any coding $\tau:\mathcal{A}\to\{0,1,\ldots,p-1\}$, the real number $\tau(\psi^\omega(a))_p$ satisfies pLC. Note that since $a$ and $b$ are arbitrary letters, we can consider them to be letters in $\{0,1,\ldots,p-1\}$ and forget the coding. By the same argument, the word $\psi^\omega(a)$ can be rewritten as a coding of a pure morphic word $w$ over the alphabet $\{0,1\}$, \textit{i.e.}, $\psi^\omega(a)=\sigma(w)$ where $\sigma(0)=a$ and $\sigma(1)=b$. If $w$ satisfies \textbf{(P2)} or \textbf{(P3)}, then $\psi^\omega(a)_p$ satisfies pLC  using the same arguments as in the proof of Theorem~\ref{Theorem:pure_morphic_2LC}.
When $\psi^\omega(a)$ is a coding of the Thue-Morse word or its complement, the situation is a bit more complicated.

Let $TM(a,b)$ be the coding of $M$, where $0$ is mapped to $a$ and $1$ is mapped to $b$. In order to complete the proof of Corollary~\ref{Cor:pure_morphic_2LC}, we will show that $TM(a,b)_p$ is well-approximable for all primes $p$ and all $a,b\in\{0,1,\ldots,p-1\}$. 

\begin{proposition}\label{TM_general_bad_approx}
Let $a,b\in\{0,1,\ldots,n-1\}$. If {$n$ is not divisible by $15$}, then $TM(a,b)_n$ is well-approximable.
\end{proposition}

\begin{proof}
We start this proof by noting that if a real number $x$ is well-approximable, then adding a rational number $p/q$ or multiplying by a rational constant will preserve this property. In particular, $TM(0,1)_n$ is well-approximable if and only if $r\cdot TM(0,1)_n$ is well-approximable for all $r\in\mathbb{Q}$. {If we restrict $r$ to} $\{0,1,\ldots,n-1\}$, then the base-$n$ expansion of $TM(0,r)_n$ is
$$r\cdot TM(0,1)_n=r\cdot\sum\limits_{i=1}^{\infty} \frac{\sigma(i)}{n^i} =  \sum\limits_{i=1}^{\infty} \frac{r\cdot\sigma(i)}{n^i} = TM(0,r)_n,$$ where $\sigma(i)$ returns the $i$-th letter in the Thue-Morse word.

Similarly, $TM(0,n-1)_n$ is well-approximable if and only if $TM(n-1,0)_n$ is well-approximable. This follows from the following observation:
$$1- TM(0,n-1)_n=1-\sum\limits_{i=1}^{\infty} \frac{(n-1)\cdot\sigma(i)}{n^i}=\sum\limits_{i=1}^{\infty} \frac{(n-1)\cdot(1-\sigma(i))}{n^i}=TM(n-1,0)_n.
$$
Multiplying by $k/(n-1)$ shows that the number $TM(k,0)_n$ is well-approximable for all $k\in\{0,1,\ldots,n-1\}$ if and only if $TM({n-1,0})_n$ is well-approximable.

Furthermore, we note that for $\ell\in\{0,1,\ldots,n-1\}$ the real number whose base-$n$ expansion is an infinite string of $\ell$'s corresponds to the rational number $\ell/(n-1)$. Therefore, if $\ell\leq n-1-k$, then
$$TM(0,k)_n+\frac{\ell}{n-1}=\sum\limits_{i=1}^{\infty} \frac{k\cdot\sigma(i)}{n^i} +\sum\limits_{i=1}^{\infty} \frac{\ell}{n^i}=\sum\limits_{i=1}^{\infty} \frac{k\cdot\sigma(i)+\ell}{n^i}=TM(\ell,\ell+k)_n.$$
Likewise, $TM(k,0)_n+\ell/(n-1)=TM(k+\ell,\ell)_n$. This, combined with the previous arguments, shows that for all $a,b\in\{0,1,\ldots,n-1\}$ the real number $TM(a,b)_n$ is well-approximable if and only if $TM(0,1)_n$ is well-approximable. Applying Theorem~\ref{Theorem:TM_bad_approx} completes the proof.
\end{proof}

\subsection{Proof of Theorem~\ref{Theorem:bounded_overlap_free_2LC}}\label{Proof:bounded_overlap_free_2LC}

Let $\mu$ be the Thue-Morse morphism. In order to prove Theorem~\ref{Theorem:bounded_overlap_free_2LC}, we will use Proposition~\ref{Proof:Subwords}, Theorem~\ref{Theorem:TM_bad_approx}, and  the following two lemmas.

\begin{lemma}\label{Lemma:finiteoverlapfree}
For every overlap-free word $x\in\{0,1\}^*$, there exist words $u,v,y\in\{0,1\}^*$ with $\length{u},\length{v}\leq{2}$ and $x=u\mu(y)v$.
\end{lemma}

\begin{lemma}\label{Lemma:overlapfreeiff}
Let $y\in\{0,1\}^*$. Then $y$ is overlap-free if and only if $\mu(y)$ is overlap-free.
\end{lemma}

For Lemma~\ref{Lemma:finiteoverlapfree}, see \cite[Theorem~6.4]{Kobayashi:1986} or \cite[Lemma~3]{ACS:1998}. For Lemma~\ref{Lemma:overlapfreeiff}, see \cite[Lemma~1.7.4]{AS:2003}.

\begin{proof}[Proof of Theorem~\ref{Theorem:bounded_overlap_free_2LC}]
 In order to prove this result, we will show that every overlap-free base-$2$ expansion of length $K$ contains a prefix of $M$ or $\widetilde{M}$ of length $p(K)$, where $\lim_{K\to\infty}p(K)=\infty$. The result then follows from Proposition~\ref{Prop:LimitWord} and Theorem~\ref{Theorem:TM_bad_approx}. 

Let $x$ be an overlap-free word of length $K$. By Lemma~\ref{Lemma:finiteoverlapfree}, there exist words $u_1,v_1,y_1\in\{0,1\}^*$ with $\length{u_1},\length{v_1}\leq{2}$ and $x=u_1\mu(y_1)v_1$. Using this construction, we can conclude that $\length{\mu(y_1)}={K-\length{u}-\length{v}}\geq{}K-4$. Furthermore, since $\mu$ is $2$-uniform, \textit{i.e.}, $\length{\mu(0)}=\length{\mu(1)}=2$, the length of $y_1$ is equal to $\length{\mu(y_1)}/2$. Provided that $K-4\geq{1}$, we also have that $y_1$ is not the empty word.

Since $x$ is overlap-free, it follows that $\mu(y_1)$ is overlap-free. Additionally, Lemma~\ref{Lemma:overlapfreeiff} implies that $y_1$ is  overlap-free. As a result, there exist $u_2,v_2,y_2\in\{0,1\}^*$ with $\length{u_2},\length{v_2}\leq{}2$ such that $y_1=u_2\mu(y_2)v_2$. Then $x$ can be rewritten as
\begin{equation*}
x=u_1\mu(y_1)v_1=u_1\mu(u_2)\mu^2(y_2)\mu(v_2)v_1.
\end{equation*}
{ The length of $y_2$ is bounded as follows:}
\begin{equation*}
\frac{\length{y_1}-4}{2}\leq |y_2|\leq \frac{
|y_1|}{2}
\end{equation*}

{More generally, for any $k\in\mathbb{N}$, the subword $y_{k}$ can be rewritten as}
$$y_k=u_{k+1}\mu(y_{k+1})v_{k+1},$$ 
{where $u_{k+1},v_{k+1},y_{k+1}\in\{0,1\}^*$, $|u_{k+1}|,|v_{k+1}|\leq{2}$ and} $$\frac{|y_k|-4}{2}\leq|y_{k+1}|\leq \frac{|y_k|}{2}.$$ {Note that  $u_{k+1},v_{k+1}$ and $y_{k+1}$ can all be the empty word.}

{Using this substitution, $x$ can be rewritten in terms of $y_{k}$ as}
\begin{equation*}
x=u_1\mu(u_2)\ldots\mu^{k-1}(u_k)\mu^k(y_k)\mu^{k-1}(v_k)\ldots\mu(v_2)v_1,
\end{equation*}
%
%

{where the length of $y_k$ is bounded below:}
\begin{equation}\label{TMprefix}
\length{y_k}\geq{\frac{K-4\cdot(2^{k}-1)}{2^k}}.
\end{equation}

From (\ref{TMprefix}), the word $y_k$ is non-empty provided that $K-4\cdot(2^k-1)>{0}$. By rearranging, the largest value of $k$ that guarantees that $y_k$ is non-empty is $k=\lfloor\log_2(K+4)\rfloor-2$. For such a value of $k$, let $a$ be any subword of $y_k$ of length $1$. Then $\mu^k(a)$ is a prefix of either $M$ or $\widetilde{M}$. Since $\mu$ is 2-uniform and $\length{a}={1}$, it follows that the length of this prefix is
\begin{align*}
\length{\mu^k(a)} &\geq{2^k}\\
&\geq{2^{(\log_2(K+4)-3)}}\\
&=\frac{K+4}{8}.
\end{align*}
Since $\lim_{K\to\infty}(K+4)/8=\infty$, the result follows.
\end{proof}

{\subsection*{Acknowledgement}
We thank the referee for a thorough reading and helpful comments.
}

{\subsection*{Funding}
Dr Blackman was supported by the Engineering and Physical Sciences Research Council (EPSRC) [grant no. EP/W006863/1] awarded through the University of Liverpool. Prof. Kristensen's research was supported by the Independent Research Fund Denmark (Grant ref. 1026-00081B) and Aarhus University Research Foundation (Grant ref. AUFF-E-2021-9-20). Dr Northey was supported by the Engineering and Physical Sciences Research Council  [grant no. EP/RF060349] awarded through Durham University.
}

\end{document}